\documentclass{amsart}
\usepackage{tikz}
\usepackage{hyperref}   
\usetikzlibrary{positioning}
\usepackage{amssymb,xcolor}
\usepackage{stmaryrd} % \owedge

\usepackage{mathtools}  
\mathtoolsset{showonlyrefs}
\theoremstyle{plain}

\theoremstyle{plain}
\newtheorem{theorem}{Theorem}[section]
\theoremstyle{plain}
\newtheorem{lemma}[theorem]{Lemma}
\theoremstyle{remark}
\newtheorem{remark}[theorem]{Remark}
\theoremstyle{plain}
\newtheorem{proposition}[theorem]{Proposition}
\numberwithin{equation}{section}

\begin{document}

\title{Scalar curvature, mean curvature and harmonic maps to the circle}

\author{Xiaoxiang Chai}
\address{Korea Institute for Advanced Study, Seoul 02455, South Korea}
\email{xxchai@kias.re.kr}

\author{Inkang Kim}
\address{Korea Institute for Advanced Study, Seoul 02455, South Korea}
\email{inkang@kias.re.kr}

\begin{abstract}
  We study harmonic maps from a 3-manifold with boundary to $\mathbb{S}^1$ and
  prove a special case of Gromov dihedral rigidity of three dimensional cubes
  whose dihedral angles are $\pi / 2$. Furthermore we give some applications
  to mapping torus hyperbolic 3-manifolds.
\end{abstract}

{\maketitle}

\footnotetext[1]{\slshape{Date:} \today }
\footnotetext[4]{Research
of Xiaoxiang Chai is supported by KIAS Grants under the research code MG074401
and MG074402.}
\footnotetext[5]{Research by Inkang Kim is partially supported by Grant
NRF-2019R1A2C1083865 and KIAS Individual Grant (MG031408).}

{\tableofcontents}

\section{Introduction}

Stern {\cite{stern-scalar-2020}} showed an interesting formula relating the
level set of harmonic functions and the scalar curvature of a 3-manifold. With
Bray {\cite{bray-scalar-2019}}, they generalized the result to 3-manifold with
boundary where a harmonic 1-form with vanishing normal component along the
boundary was studied. Different from {\cite{bray-scalar-2019}}, we send the
boundary to a point i.e. we study a Dirichlet boundary condition. We obtain a
similar formula, and then we combine the technique with
{\cite{stern-scalar-2020,bray-scalar-2019}} to study the dihedral rigidity of
standard cubes in $\mathbb{R}^3$.

\begin{theorem}
  \label{dirichlet condition}For a harmonic map $u : N^3 \to \mathbb{S}^1$
  with a Dirichlet condition $u|_{\partial N} = [0] \in \mathbb{S}^1$ and for
  almost all $\theta \in \mathbb{S}^1$ the level set $\Sigma_{\theta} = u^{-
  1} (\theta)$ being closed, we have the inequality
  \begin{equation}
    \int_N \tfrac{1}{2} [\tfrac{| \mathrm{Hess} u|^2}{| \mathrm{d} u|} + R_N |
    \mathrm{d} u|_{}] + \int_{\partial N} H_{\partial N} | \mathrm{d} u|
    \leqslant 2 \pi \int_{\mathbb{S}^1} \chi (\Sigma_{\theta}),
    \label{dirichlet inequality}
  \end{equation}
  where $R_N$ is the scalar curvature of $N$ and $H_{\partial N}$ is the mean
  curvature of $\partial N$.
\end{theorem}

\begin{remark}
  The first author learned from Pengzi Miao that the theorem appeared in
  general form {\cite[Proposition 2.2]{hirsch-mass-2020}} for harmonic
  functions.
\end{remark}

This result and the one in {\cite{bray-scalar-2019}} suggest that it is
worthwhile to consider harmonic maps to $\mathbb{S}^1$ with mixed boundary
conditions. We study the problem on three dimensional cubes. We identify
$\mathbb{S}^1$ with $\mathbb{R}/\mathbb{Z}$ and use $[r]$, $r \in \mathbb{R}$
to denote an element of $\mathbb{S}^1$. Let $(Q^3, g)$ be a Riemannian
manifold diffeomorphic to a standard unit cube in $\mathbb{R}^3$, $T$ and $B$
be respectively the top and bottom face, $F$ be the union of side faces, $\nu$
be the outward unit normal to each face. The angle $\gamma$ formed by
neighboring two unit normals is called exterior dihedral angle and $\pi -
\gamma$ is called the dihedral angle between two neighboring faces. We assume
that the dihedral angles are everywhere equal to $\pi / 2$. For the general
case, we can use the bending construction of Gromov
({\cite{gromov-metric-2018}}, {\cite{li-dihedral-2020}}) to reduce to the case
where dihedral angles are everywhere equal to $\pi / 2$.

\begin{theorem}[{\cite{li-polyhedron-2020}}, {\cite{li-dihedral-2020}}]
  \label{dihedral}Unless $(Q^3, g)$ is isometric (up to constant multiple of
  the metric) to the standard Euclidean cube, the following three conditions
  cannot be satisfied at the same time on the cube $(Q^3, g)$:
  \begin{enumerate}
    \item The scalar curvature of $Q^3$ is nonnegative;
    
    \item Every face of $Q^3$ is mean convex;
    
    \item The dihedral angle is equal to $\pi / 2$ everywhere along each edge.
  \end{enumerate}
  We assume that the metric is $C^{2, \alpha}$ for some $\alpha \in (0, 1)$.
\end{theorem}

This theorem was conjectured as the \text{{\itshape{dihedral rigidity}}} by
Gromov {\cite{gromov-dirac-2014}} and was verified by Li
{\cite{li-polyhedron-2020}}, {\cite{li-dihedral-2020}} using free boundary
minimal surfaces. {\color{blue}{We solve the Laplace equation with Dirichlet
condition on bottom and top faces, and with Neumann boundary condition along
side faces, then we apply the level set method of Stern. We address the
regularity of the problem in the Appendix.}}

We mention two interesting questions. For the cone type polyhedron, it seems
desirable to find a boundary condition which is analogue to the capillary
condition of the minimal surface. The second question is about the spacetime
version of the problem. Recently {\cite{bray-harmonic-2019}} proved the
Riemannian positive mass theorem using harmonic coordinate method, which later
was generalized to spacetime by {\cite{hirsch-spacetime-2020}}. Recall that a
function $u \in C^2 (Q)$ is called \text{{\itshape{spacetime harmonic}}} if
\[ \Delta_g u + \mathrm{tr}_g k | \nabla u| = 0\ \text{in } Q, \label{spacetime
   harmonic} \]
where $k$ is a prescribed symmetric 2-tensor. The natural boundary condition
is $u = 0$ along bottom and top faces of $Q$ and $\tfrac{\partial u}{\partial
\nu} = 0$ along side faces. A preliminary calculation shows that with these
mixed boundary conditions it will give a spacetime version of the
{\cite{li-polyhedron-2020}} if we can control the level set topology and
assuming corresponding convexity conditions.

In the last section, we give one application to a hyperbolic 3-manifold, a
mapping torus by a pseudo-Anosov map. Usually, people use Kleinian group
theory to attack such a problem, but we use a simple differential geometric
approach of harmonic maps.

\begin{theorem}
  \label{main1}Let $M_{\phi}$ be a mapping torus of a closed surface $S$ of
  genus $g \geq 2$ via a pseudo-Anosov map $\phi$. Then
  \[ g \leqslant \frac{3}{4 C\pi \| \phi \|} \mathrm{vol} (M_{\phi}) + 1, \]
  where $\| \phi \|$ is the translation length of a hyperbolic isometry
  defined by $\phi$ and $C$ is a constant depending only  on $S$ and the injectivity
  radius of $M_\phi$.
\end{theorem}

It seems there would be more applications to 3-manifold topology as in
{\cite{bray-scalar-2019}}, which should be explored more in a near future.

\text{{\bfseries{Acknowledgement}}} The first author would like to thank
Martin Li (CUHK) and Shanjiang Chen (CUHK) for discussions on mixed boundary
value problems, and Pengzi Miao (Miami University) for the reference
{\cite{hirsch-mass-2020}}. The second author thanks Y. Minsky for the
correspondence about the geometric model of the mapping torus, and S. Kojima
for answering questions about his work.

\section{Formula for Laplacian of energy of harmonic map to the circle}

In this section we collect some basic formulas for Laplacian of energy of
harmonic map to the circle. The reference is {\cite{stern-scalar-2020}}. Let
$u : N \to \mathbb{S}^1$ be a harmonic map from a closed Riemannian 3-manifold
to the unit circle.

Choose an orthonormal frame $e_1, e_2, e_3$ adapted to $\Sigma_{\theta} = u^{-
1} (\theta)$, so that $e_1, e_2$ are tangential to $\Sigma_{\theta}$, and $e_3
= \frac{\nabla u}{| \nabla u|}$. Let $R_{ij}$ denote the sectional curvature
of $N$ for the section $e_i \wedge e_j$. The symmetric quadratic tensor
$(h_{ij} = \langle D_{e_i} e_3, e_j \rangle)$ is the second fundamental form
$k_{\Sigma_{\theta}}$ for $\Sigma_{\theta}$. Note that $k_{\Sigma_{\theta}} =
(| \nabla u|^{- 1} D \mathrm{d} u) |_{\Sigma_{\theta}}$.

Then Gauss equation gives
\[ K = R_{12} + h_{11} h_{22} - h_{12}^2, \]
and the scalar curvature $R_N$ of $N$ is
\[ R_N = 2 (R_{12} + R_{13} + R_{23}) \]
and
\[ \mathrm{Ric} (e_3, e_3) = R_{13} + R_{23} . \]
The mean curvature $H_{\Sigma_{\theta}} = \mathrm{tr} k_{\Sigma_{\theta}}$ is
$h_{11} + h_{22}$ and the scalar curvature $R_{\Sigma_{\theta}}$ is $2 K$.
Hence
\[ \mathrm{Ric} (e_3, e_3) = \frac{1}{2}  (R_N - R_{\Sigma_{\theta}} +
   H^2_{\Sigma_{\theta}} - |k_{\Sigma_{\theta}} |^2) . \]
Using harmonicity of $u$, one can verify that
\[ | \nabla u|^2  (H^2_{\Sigma_{\theta}} - |k_{\Sigma_{\theta}} |^2) = 2 |
   \mathrm{d} | h||^2 - |Dh|^2, \]
and one can rewrite
\begin{align}
& \mathrm{Ric} (\nabla u, \nabla u) \\
= & | \nabla u|^2 \mathrm{Ric} (e_3, e_3) \\
= & \frac{1}{2}  | \nabla u|^2  (R_{\Sigma_{\theta}} - R_{\Sigma_{\theta}})
+ \frac{1}{2}  (2| \mathrm{d} | \nabla u||^2 - |D \mathrm{d} u|^2) .
\label{schoen yau}
\end{align}
Then using the standard Bochner identity for $\mathrm{d} u$
\begin{equation}
  \Delta_g  \frac{1}{2}  | \nabla u|^2 = |D \mathrm{d} u|^2 + \mathrm{Ric}
  (\mathrm{d} u, \mathrm{d} u), \label{bochner harmonic one form}
\end{equation}
one can deduce the formula as in {\cite{stern-scalar-2020}}
\begin{equation}
  2 \int_N \frac{| \mathrm{d} u|}{2} R_{\Sigma_{\theta}} = 4 \pi \int_{\theta}
  \chi (\Sigma_{\theta}) \mathrm{d} \theta \geqslant \int_N R_N | \mathrm{d}
  u|, \label{average}
\end{equation}
if $N$ is closed where the first equality follows from the coarea formula and
Gauss-Bonnet theorem. Let $\varphi_{\delta} = \sqrt{|h|^2 + \delta}$ with
$\delta > 0$, it follows that
\begin{equation}
  \Delta \varphi_{\delta} = \tfrac{1}{\varphi_{\delta}} [\tfrac{1}{2} \Delta
  |h|^2 - \tfrac{|h|^2}{\varphi_{\delta}^2} | \mathrm{d} |h||^2] \geqslant
  \tfrac{1}{\varphi_{\delta}} [|Dh|^2 - | \mathrm{d} |h||^2 + \mathrm{Ric} (h,
  h)] .
\end{equation}
Inserting \eqref{schoen yau} and Bochner formula \eqref{bochner harmonic one
form} into the above, we have that along regular level sets $\Sigma_{\theta}$
of $u$,
\begin{equation}
  \Delta_g \varphi_{\delta} \geqslant \tfrac{1}{2 \varphi_{\delta}} [|
  \mathrm{Hess} u|^2 + | \mathrm{d} u|^2 (R_N - R_{\Sigma_{\theta}})] .
  \label{schoen-yau}
\end{equation}

\section{Application to rigidity of scalar curvature and mean curvature}

Let $N$ be a 3-manifold with boundary $\partial N \neq \emptyset$. We require
that $H_2 (N ; \mathbb{Z})$ is nontrivial. Let $\alpha$ be a nontrivial
element of $H_2 (N ; \mathbb{Z})$, then according to Poincar{\'e}-Lefschetz
duality that $H^1 (N, \partial N ; \mathbb{Z})$ is isomorphic to $H_2 (N,
\mathbb{Z})$. Let $\alpha^{\ast}$ be the corresponding element of $H^1 (N,
\partial N ; \mathbb{Z})$ under this isomorphism. \ Let $[N, \partial N :
\mathbb{S}^1]$ be the homotopy classes of maps from $N$ to $\mathbb{S}^1$
sending the boundary $\partial N$ to a point. Hence $\alpha$ determines a
nontrivial homotopy class in $[\tilde{u}] \in [N, \partial N : \mathbb{S}^1]$.
We minimize the energy in this homotopy class, we obtain a harmonic map $u \in
[\tilde{u}]$ satisfying the conditions in Theorem \ref{dirichlet condition}.
The Hodge-Morrey theory {\cite[Chapter 5]{giaquinta-cartesian-1998}} applied
to the relative cohomology class $[\tilde{u}^{\ast} (\mathrm{d} \theta)] \in
H^1 (N, \partial N ; \mathbb{Z})$ yields an energy minimizing representative
$u : N \to \mathbb{S}^1$.

\begin{proof}[Proof of Theorem \ref{dirichlet condition}]
  Main computation was done already in {\cite{stern-scalar-2020}}. Note that
  every level set $u^{- 1} (\theta)$ does not intersect the boundary $\partial
  N$ except at the level $[0] \in \mathbb{S}^1$. Let $h = u^{\ast} (\mathrm{d}
  \theta)$ be gradient 1-form, so $h$ is harmonic. Let $\varphi_{\delta} =
  \sqrt{|h|^2 + \delta}$ with $\delta > 0$. From
  {\cite[(14)]{stern-scalar-2020}}, we have that along regular level sets
  $\Sigma$ of $u$,
  \begin{equation}
    \Delta_g \varphi_{\delta} \geqslant \tfrac{1}{2 \varphi_{\delta}} [|
    \mathrm{Hess} u|^2 + | \mathrm{d} u|^2 (R_N - R_{\Sigma})] .
  \end{equation}
  Let $\mathcal{A} \subset \mathbb{S}^1$ be an open set containing the
  critical values of $u$ and let $\mathcal{B}$ be the complement subset. So
  $\mathcal{B}$ contains only regular values. We have
  \begin{equation}
    \int_{\partial N} \tfrac{\partial \varphi_{\delta}}{\partial \nu} =
    \int_{u^{- 1} (\mathcal{A})} \Delta_g \varphi_{\delta} + \int_{u^{- 1}
    (\mathcal{B})} \Delta_g \varphi_{\delta} . \label{integration by parts}
  \end{equation}
  Since $u$ is smooth by elliptic regularity, using local coordinates of
  $\mathbb{S}^1$, $u$ is a harmonic function. By a direct calculation and
  $u|_{\partial N} = [0]$ is constant along $\partial N$, so
  \begin{equation}
    \Delta_{\partial N} u = 0, \quad \Delta_g u = 0 = \Delta_{\partial N} u +
    H_{\partial N} \langle h, \nu \rangle + \mathrm{Hess} (u) (\nu, \nu) .
  \end{equation}
  The above equality can be deduced via the following: Let $e_i$, $\nu$ be an
  orthonormal frame of $N$, then
\begin{align}
\Delta u = & \sum_{e_i / / \partial N} \mathrm{Hess}_N u (e_i, e_i) +
\mathrm{Hess} u (\nu, \nu) \\
= & \sum_{e_i / / \partial N} (e_i e_i u - \nabla^N_{e_i} e_i u) +
\mathrm{Hess} u (\nu, \nu) \\
= & \sum_{e_i / / \partial N} (e_i e_i u - \nabla^{\partial N}_{e_i} e_i
u) - \langle \nabla^N_{e_i} e_i, \nu \rangle \langle \nabla u, \nu \rangle
+ \mathrm{Hess} u (\nu, \nu) \\
= & \Delta_{\partial N} u + H_{\partial N} \langle h, \nu \rangle +
\mathrm{Hess} (u) (\nu, \nu) .
\end{align}
  Also, we have that
\begin{align}
\tfrac{\partial \varphi_{\delta}}{\partial \nu} & =
\tfrac{1}{\varphi_{\delta}} \mathrm{Hess} (u) (h, \nu) \\
& = \tfrac{\langle h, \nu \rangle}{\varphi_{\delta}} \mathrm{Hess} (u)
(\nu, \nu) \\
& = - \tfrac{|h|^2}{\varphi_{\delta}} H_{\partial N} .  \label{boundary
mean curvature}
\end{align}
  We have also globally on $N$ that (see {\cite{stern-scalar-2020}})
  \begin{equation}
    \Delta_g \varphi_{\delta} \geqslant - C_N |h|
  \end{equation}
  for some constant $C_N > 0$ depending only on $N$. We see that
  \begin{equation}
    - \int_{u^{- 1} (\mathcal{A})} \Delta \varphi_{\delta} \leqslant C_N 
    \int_{u^{- 1} (\mathcal{A})} |h| = C_N  \int_{\mathcal{A}} |
    \Sigma_{\theta} |, \label{critical estimate interior}
  \end{equation}
  where we have applied coarea formula. So we have from \eqref{schoen-yau},
  \eqref{integration by parts} and taking limits as $\delta \to 0$,
\begin{align}
& \int_{u^{- 1} (\mathcal{B})} \tfrac{| \mathrm{d} u|}{2} (\tfrac{|
\mathrm{Hess} u|^2}{| \mathrm{d} u|^2} + R_N - R_{\Sigma}) \\
\leqslant & \lim_{\delta \to 0} \left[ \int_{\partial N} \tfrac{\partial
\varphi_{\delta}}{\partial \nu} - \int_{u^{- 1} (\mathcal{A})} \Delta_g
\varphi_{\delta} \right] \\
\leqslant & - \int_{\partial N} H_{\partial N}  | \mathrm{d} u| + C_N
\int_{\mathcal{A}} | \Sigma_{\theta} | .
\end{align}
  Rearranging and applying the coarea formula once again, we have that
\begin{align}
& \int_{u^{- 1} (\mathcal{B})} \tfrac{| \mathrm{d} u|}{2} (\tfrac{|
\mathrm{Hess} u|^2}{| \mathrm{d} u|^2} + R_N) + \int_{\partial N}
H_{\partial N} | \mathrm{d} u| \\
\leqslant & \tfrac{1}{2} \int_{\theta \in \mathcal{B}}
\int_{\Sigma_{\theta}} R_{\Sigma} + C_N  \int_{\mathcal{A}} |
\Sigma_{\theta} | \\
= & 2 \pi \int_{\theta \in \mathcal{B}} \chi (\Sigma_{\theta}) + C_N
\int_{\mathcal{A}} | \Sigma_{\theta} | .
\end{align}
  In the last line, we have used Gauss-Bonnet theorem. Since $\mathcal{H}^1
  (\mathcal{A})$, the Hausdorff measure, can be made arbitrarily small using
  the Sard theorem and $\theta \mapsto | \Sigma_{\theta} |$ is integrable over
  $\mathbb{S}^1$ by coarea formula, sending $\mathcal{H}^1 (\mathcal{A})$ to
  zero leads to our inequality \eqref{dirichlet inequality}.
\end{proof}

Now we discuss Theorem \ref{dihedral}. Let $(x_1, x_2, x_3)$ be the
coordinates on $Q = Q^3$ induced by the diffeomorphism $\Phi$ from $[0, 1]^3$,
and we identify the top face with $\{x_3 = 1, 0 \leqslant x_1, x_2 \leqslant
1\}$ and the bottom face with $\{x_3 = 0, 0 \leqslant x_1, x_2 \leqslant 1\}$.
We consider the class of maps which are homotopic to the map $\tilde{u} : Q
\to \mathbb{S}^1$ given by
\begin{equation}
  \tilde{u}  (x_1, x_2, x_3) = [x_3]
\end{equation}
and takes the value $[0] \in \mathbb{S}^1$ at $x_3 = 0$ and $x_3 = 1$. By
Poincar{\'e}-Lefschetz duality,
\begin{equation}
  H^1 (Q, T \cup B ; \mathbb{Z}) \cong H_2 (Q, F ; \mathbb{Z}) .
\end{equation}
One can associate the homotopy class $[\tilde{u}]$ to an element $\alpha \in
H_2 (Q, F ; \mathbb{Z})$ where $\tilde{u}^{- 1} (\theta)$ represents $\alpha$.

We do minimization of the energy in this homotopy class. The Hodge-Morrey
theory {\cite[Chapter 5]{giaquinta-cartesian-1998}} slightly modified to mixed
boundary conditions applied to the relative cohomological class
$[\tilde{u}^{\ast} (\mathrm{d} \theta)] \in H^1 (Q, T \cup B ; \mathbb{Z})$
yields an energy minimizing representative $u : Q \to \mathbb{S}^1$ with
Sobolev regularity. The existence of a harmonic map is equivalent to the
existence of a solution $u$ to the following mixed boundary value problem
\begin{equation}
  \Delta_g u = 0\ \text{in } Q, \tfrac{\partial u}{\partial \nu} = 0\
  \text{along } F, u = 1\ \text{on } T, u = 0\ \text{on } B. \label{mvp}
\end{equation}
Indeed, if $0 \leqslant u \leqslant 1$ in $\bar{Q}$, after identifying 0 and
1, the solution to the above gives a harmonic map $u$ from $Q$ to
$\mathbb{S}^1$. In fact, we can show that $u \in C^{2, \alpha} (\bar{Q},
\mathbb{S}^1)$ (See Theorem \ref{general existence}).

\begin{proof}[Proof of Theorem \ref{dihedral}]
  The harmonic map equation reduces to the harmonicity of the pull back 1-form
  $h = u^{\ast} (\mathrm{d} \theta)$. That is, the harmonic map equation gives
  \begin{equation}
    \mathrm{d} h = 0, \quad \mathrm{d}^{\ast} h = 0\ \text{in } Q.
    \label{harmonic 1-form}
  \end{equation}
  Since $u$ takes fixed values at top and bottom faces, the 1-form $h$ has
  only a normal component along $T$ and $B$. The Dirichlet boundary condition
  gives
  \begin{equation}
    h \wedge \nu = 0\ \text{on } T \cup B. \label{dirichlet}
  \end{equation}
  And $h$ satisfies the Neumann condition
  \begin{equation}
    \langle h, \nu \rangle = 0\ \text{on } F. \label{neumann}
  \end{equation}
  From Sard's theorem, the level set $u^{- 1} (\theta)$ is a $C^1$ submanifold
  of $Q$ and hence multiple copies of squares.
  
  Let $\mathcal{A} \subset \mathbb{S}^1$ be an open set containing the
  critical values of $u$ and let $\mathcal{B}$ be the complement subset. So
  $\mathcal{B}$ contains only regular values. From integration by parts,
  \begin{equation}
    \int_{T \cup B} \tfrac{\partial \varphi_{\delta}}{\partial \nu} + \int_F
    \tfrac{\partial \varphi_{\delta}}{\partial \nu} = \int_{u^{- 1}
    (\mathcal{A})} \Delta_g \varphi_{\delta} + \int_{u^{- 1} (\mathcal{B})}
    \Delta_g \varphi_{\delta} . \label{integration by parts cube}
  \end{equation}
  Similar to \eqref{boundary mean curvature}, we have
  \begin{equation}
    \int_{T \cup B} \tfrac{\partial \varphi_{\delta}}{\partial \nu} = -
    \int_{T \cup B} H_{\partial N} \tfrac{|h|^2}{\varphi_{\delta}} \to -
    \int_{T \cup B} H_{\partial N} | \mathrm{d} u| \label{mean curvature top
    and bottom}
  \end{equation}
  as $\delta \to 0$. Similar to {\cite{bray-scalar-2019}}, with
  $\tfrac{\partial \varphi_{\delta}}{\partial \nu} = \langle \mathrm{d}
  \varphi_{\delta}, \nu \rangle = - \varphi_{\delta}^{- 1} \langle h, D_h \nu
  \rangle$ using the Neumann condition on $F$, and $(\kappa_{\partial
  \Sigma_{\theta}} - H_{\partial N}) | \mathrm{d} u| = - |h|^{- 1} \langle h,
  D_h \nu \rangle$ we have that
\begin{align}
& \lim_{\delta \to 0}  \int_F \tfrac{\partial \varphi_{\delta}}{\partial
\nu} \\
= & \lim_{\delta \to 0} [\int_{F \cap u^{- 1} (\mathcal{A})}
\tfrac{\partial \varphi_{\delta}}{\partial \nu} + \int_{F \cap u^{- 1}
(\mathcal{B})} \tfrac{\partial \varphi_{\delta}}{\partial \nu}] \\
\leqslant & C_Q  \int_{\theta \in \mathcal{A}} | \partial \Sigma_{\theta}
| + \int_{F \cap u^{- 1} (\mathcal{B})} (\kappa_{\partial \Sigma_{\theta}}
- H_{\partial N}) | \mathrm{d} u|  \label{mean curvature side}
\end{align}
  as $\delta \to 0$. From \eqref{schoen-yau}, \eqref{critical estimate
  interior} and \eqref{integration by parts cube},
  \begin{equation}
    \int_{u^{- 1} (\mathcal{B})} \tfrac{| \mathrm{d} u|}{2} (\tfrac{|
    \mathrm{Hess} u|^2}{| \mathrm{d} u|^2} + R_Q - R_{\Sigma}) \leqslant
    \int_{T \cup B} \tfrac{\partial \varphi_{\delta}}{\partial \nu} + \int_F
    \tfrac{\partial \varphi_{\delta}}{\partial \nu} + C_Q  \int_A |
    \Sigma_{\theta} | .
  \end{equation}
  Inserting \eqref{mean curvature top and bottom} and \eqref{mean curvature
  side} into the above (with some re-ordering) and taking $\delta \to 0$,
  using $\tfrac{\partial \varphi_{\delta}}{\partial \nu} = - |h| H_{\partial
  Q}$ on $T \cup B$ we have
\begin{align}
& \int_{u^{- 1} (\mathcal{B})} \tfrac{| \mathrm{d} u|}{2} (\tfrac{|
\mathrm{Hess} u|^2}{| \mathrm{d} u|^2} + R_Q) + \int_{F \cap u^{- 1}
(\mathcal{B})} H_{\partial Q} | \mathrm{d} u| + \int_{T \cup B}
H_{\partial Q} | \mathrm{d} u| \\
\leqslant & \int_{u^{- 1} (\mathcal{B})} \tfrac{1}{2} R_{\Sigma}  |
\mathrm{d} u| + \int_{F \cap u^{- 1} (\mathcal{B})} \kappa_{\partial
\Sigma_{\theta}} | \mathrm{d} u| + C_Q  (\int_{\theta \in \mathcal{A}} |
\Sigma_{\theta} | + | \partial \Sigma_{\theta} |) \\
= & \int_{\theta \in \mathcal{B}} \int_{\Sigma_{\theta}} \tfrac{1}{2}
R_{\Sigma} + \int_{\theta \in \mathcal{B}} \int_{\partial \Sigma_{\theta}}
\kappa_{\partial \Sigma_{\theta}} + C_Q  (\int_{\theta \in \mathcal{A}} |
\Sigma_{\theta} | + | \partial \Sigma_{\theta} |) \\
= & \int_{\theta \in \mathcal{B}} [2 \pi \chi (\Sigma_{\theta}) - \sum_j
\gamma_j] + C_Q  (\int_{\theta \in \mathcal{A}} | \Sigma_{\theta} | + |
\partial \Sigma_{\theta} |) .
\end{align}
  Here we have used the coarea formula and the Gauss-Bonnet theorem (with
  turning angle). By definition of Euler characteristic, we have that $\chi
  (\Sigma_{\theta}) \leqslant 1$. Now we analyze the turning angle $\gamma_i$.
  Note that $\gamma_i$ is $\pi$ minus the interior turning angle. By coarea
  formula, we get the integrability of $\theta \mapsto | \Sigma_{\theta} |$.
  
  Let $F_1$ and $F_2$ be any pair of neighboring side faces, $E = F_1 \cap
  F_2$, $\nu_i$ be the normal of the face $F_i$, $\tau$ be the tangent vector
  of $E$. We pick a point $p \in E$ which is not a vertex. We analyze the
  gradient vector field $\nabla u$. Since that $\nabla u$ has no component in
  $\nu_i$ direction according to the boundary condition and the vector $\tau$
  is normal to both $\nu_1$ and $\nu_2$, so $\nabla u$ must be parallel to
  $\tau$ along $E$. Therefore, $E$ intersects the level set $u^{- 1} (\theta)$
  orthogonally. So $\nu_i$ coincides with the tangent vector of $u^{- 1}
  (\theta) \cap F_j$ along $E$, and the exterior turning angle of $u^{- 1}
  (\theta) \cap F_i$ to $u^{- 1} (\theta) \cap F_j$ is the same as the angle
  forming by $\nu_1$ and $\nu_2$. So we have that
  \begin{equation}
    \int_{\theta \in \mathcal{B}} [2 \pi \chi (\Sigma_{\theta}) - \sum_j
    \gamma_j] \leqslant \int_{\theta \in \mathcal{B}} \sum_{j = 1}^4
    (\tfrac{\pi}{2} - \gamma_j) \leqslant 0 \label{angle}
  \end{equation}
  using that the dihedral angle is everywhere equal to $\pi / 2$. Here
  $\gamma_j$ is the four turning angles. By Sard type theorem, we can take
  $\mathcal{H}^1 (\mathcal{A})$ arbitrarily small, we have that
  \begin{equation}
    \int_Q \tfrac{| \mathrm{d} u|}{2} (\tfrac{| \mathrm{Hess} u|^2}{|
    \mathrm{d} u|^2} + R_Q) + \int_{\partial Q} H_{\partial Q} | \mathrm{d} u|
    \leqslant 0.
  \end{equation}
  By nonnegativity of $R_Q$ and $H_{\partial Q}$, we have that $\mathrm{Hess}
  (u) \equiv 0$. We have that $R_Q \equiv 0$, $H_{\partial Q} \equiv 0$ and
  $\gamma_j \equiv \pi / 2$ and every regular level set intersect the vertical
  edges only four times. Fixing any component $S$ of the regular level set
  $u^{- 1} (\theta)$. The map
  \begin{equation}
    \Psi : S \times \mathbb{R} \to Q, \quad \tfrac{\partial \Psi}{\partial t}
    = \tfrac{\mathrm{grad} u}{| \mathrm{grad} u|} \circ \Psi
  \end{equation}
  gives a local isometry. Then $\Sigma$ has vanishing curvature and $\partial
  \Sigma$ has vanishing geodesic curvature. This says that $Q$ is a three
  dimensional Euclidean cube (up to a constant multiple of the metric).
\end{proof}

\section{Application to hyperbolic mapping torus}

In his geometrization program, Thurston proved that a mapping torus of a
surface of genus at least 2 by a pseudo-Anosov map is hyperbolizable. By
Mostow rigidity, this mapping torus has a unique hyperbolic structure, and
hence it has an associated hyperbolic volume. Then it is a natural question to
find the genus bound of the surface in terms of the volume. As an application
of a current technique, we give an upper bound for the genus in terms of the
volume and the hyperbolic translation length of the pseudo-Anosov map. The
main estimate is

\begin{theorem}
  \label{main}Let $M_{\phi}$ be a mapping torus of a closed surface $S$ of
  genus $g \geq 2$ via a pseudo-Anosov map $\phi$. Then
  \[ g \leqslant \frac{3}{4 C\pi \| \phi \|} \mathrm{vol} (M_{\phi}) + 1, \]
  where $\| \phi \|$ is the translation length of a hyperbolic isometry
  defined by $\phi$ and $C$ is a constant depending only  on $S$ and the injectivity
  radius of $M_\phi$.
\end{theorem}

In the last section, we compare $\| \phi \|$ to the entropy of the
pseudo-Anosov map $\phi$. The entropy of a pseudo-Anosov map is $\log \lambda
(\phi)$ where $\lambda (\phi)$ is the dilatation of $\phi$. Another
interpretation of the entropy is the translation length of the action of
$\phi$ on the Teichm{\"u}ller space with respect to the Teichm{\"u}ller
metric.

Using Minsky's geometric model, one can show that the entropy $\mathrm{ent}
(\phi)$ and $\| \phi \|$ are comparable once $S$ is fixed and the injectivity
radius of the mapping torus $M_{\phi}$ is bounded below. Hence we obtain

\begin{theorem}
  \label{thm2}{$\mbox{}$}
  
  \[ \frac{1}{3 \pi | \chi (S) |} \mathrm{ent} (M_{\phi}) \leqslant
     \mathrm{ent} (\phi) \leqslant \frac{3}{2 \pi | \chi (S) | K} \mathrm{ent}
     (M_{\phi}), \]
  where $K$ depends only on $S$ and the injectivity radius of $M_{\phi}$.
\end{theorem}

In general, if the injectivity radius goes to zero, $K$ also tends to zero.
Since there exist families of pseudo-Anosov maps whose entropy tends to
infinity while the volume of the mapping torus remains bounded, this is the
best that we can hope for except the explicit calculation of $K$. This
inequality is obtained in {\cite{kin-entropy-2009}} using Brock's inequality
{\cite{brock-weil-petersson-2003}}. Our proof relies on the harmonic map
technique in {\cite{stern-scalar-2020}}, and it is simpler.

\subsection{Genus bound for mapping torus}

Let $M_{\phi}$ be a hyperbolic mapping torus of $S$ via a pseudo-Anosov map
$\phi$ and
\[ u : M_{\phi} = S \times [0, 1] / (x, 0) \sim (\phi (x), 1) \to [0, 1] / 0
   \sim 1 \]
the projection. On the infinite cyclic cover $S \times \mathbb{R}$, $\phi$
acts as a translation. Since $\pi_1 (M_{\phi}) = \langle \pi_1 (S), t|t \gamma
t^{- 1} = \phi_{\ast} (\gamma), \gamma \in \pi_1 (S) \rangle$, $\phi$
corresponds to a hyperbolic isometry $t$, and we denote $\| \phi \|$ the
translation length of $t$ on $\mathbb{H}^3$. This is the hyperbolic
translation length of $\phi$ on the infinite cyclic cover $S \times
\mathbb{R}$. Hence $\| \phi \|$ denotes the width of the fundamental domain of
$M_{\phi}$ on the cyclic cover $S \times \mathbb{R}$ where the left and right
sides are identified by the action of $\phi$.

\begin{figure}[h]
  \begin{center}
    \begin{tikzpicture}[x=1cm,y=1cm]

\begin{scope}[shift={(2,0)}, thick]
\clip(-1.8,-2)rectangle(3,2);
\draw (0,0) circle [x radius=0.6, y radius=1.3];
\end{scope}
\begin{scope}[shift={(2,0)}, thick]
\clip(-1.8,-2)rectangle(3,2);
\draw (0,0.6) circle [x radius=0.1, y radius=0.5];
\end{scope}
\begin{scope}[shift={(2,0)}, thick]
\clip(-1.8,-2)rectangle(3,2);
\draw (0,-0.6) circle [x radius=0.1, y radius=0.5];
\end{scope}
\begin{scope}[shift={(-3,0)}, thick]
\clip(-1.8,-2)rectangle(3,2);
\draw (0,0) circle [x radius=0.6, y radius=1.3];
\end{scope}

\begin{scope}[shift={(-3,0)}, thick]
\clip(-1.8,-2)rectangle(3,2);
\draw (0,0.6) circle [x radius=0.1, y radius=0.5];
\end{scope}
\begin{scope}[shift={(-3,0)}, thick]
\clip(-1.8,-2)rectangle(3,2);
\draw (0,-0.6) circle [x radius=0.1, y radius=0.5];
\end{scope}
%\draw[shift={(2,0)}, yscale=cos(70), thick] (-1,0) arc (-180:0:1);
%\draw[shift={(2,0)}, yscale=cos(70), thick] (-0.7,-0.6) arc (180:0:0.7);

\draw[thick] (-4,-1.3) -- (4,-1.3);
\draw[thick] (-4,1.3) -- (4,1.3);
\draw[thick] (-2, -3) -- (1.5, -3);
%\draw [xscale=cos(70), dashed, thick] (0,-0.3) arc (-90:90:0.3);
%\draw [xscale=cos(70), thick] (0,0.3) arc (90:270:0.3);

%\draw [shift={(-0.7,0)}, xscale=cos(70), dashed, thick] (0,-0.3) arc (-90:90:0.3);
%\draw [shift={(-0.7,0)}, xscale=cos(70), thick] (0,0.3) arc (90:270:0.3);

%\draw [thick] (0,0.3) .. controls (0.1,0.3) and (0.1,0.34) .. (0.21,0.446);
%\draw [thick] (0,-0.3) .. controls (0.1,-0.3) and (0.1,-0.34) .. (0.21,-0.446);
\draw (-2, -2.8) node{{0}};
\draw (1.5, -2.8) node{{1}};
\draw(0.5, -2) node{{$u$}};
\draw(0, -2) node{{$\downarrow$}};
\draw (-0.5, -0.6) node{{$\Longrightarrow$}};
\draw (4, 0) node{{$S\times {\mathbb R}$}};
\draw (0,0) node{{$\phi$}};
\draw (-2.3,0) node{{$\rightarrow$}};
\draw (-1.85,0) node{{${\nabla u}$}};
\draw (-3, -1.6) node{{$u^{-1}(0)$}};
\draw (2, -1.6) node{{$u^{-1}(1)$}};
\draw (-5,0) node{{\tiny $-\infty$}};
%\draw (-0.7,-0.5) node{{\tiny $0$}};
%\draw (-0.35,-0.6) node{{\tiny $I$}};
%\draw (-1.1,0) node{{\tiny $\partial \Sigma$}};
%\draw (0,-0.5) node{{\tiny $1$}};
%\draw (2.7, 0.5) node{{\tiny $\Sigma$}};
%\draw (0,-1.3) node{{$\widehat{\Sigma}$}};

\end{tikzpicture}
  \end{center}
  
  \
  \caption{Lifted map of $u$ to $S \times \mathbb{R} \rightarrow [0, 1]
  \subset \mathbb{R}$.}
\end{figure}
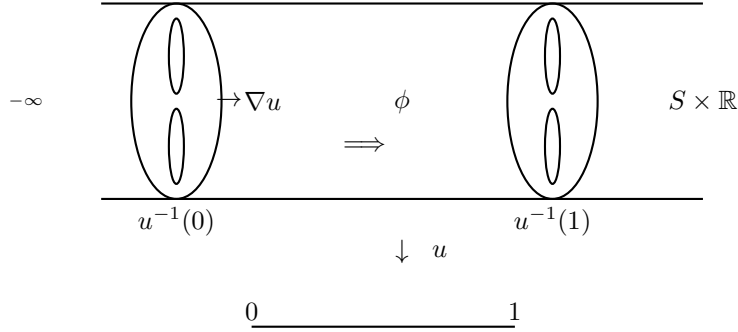

Choose orthonormal basis $e_1, e_2$ tangent to $\Sigma_{\theta} = u^{- 1}
(\theta)$, $e_3$ such that $e_3 = \frac{\nabla u}{| \nabla u|}$. %Then $|\mathrm{d} u| = | \nabla u| = \frac{1}{\| \phi \|}$.
 Note that  for any $x\in S\times \{0\}$,  along the gradient flow starting from $x$
 in  $u^{-1}(0)$ to $u^{-1}(1)$, $u$ behaves like the projection from $[0, \text{length of the gradient flow}]$ to $[0,1]$. Hence $|du|\leq \frac{1}{ \text{length of the gradient flow}}\leq \frac{1}{\text{W=width of the fundamental domain}}$ where the width means the smallest distance between points in
  $S\times\{0\}$ and in $S\times\{1\}$.
 Then $$\| \mathrm{d}
  u\|_{L^2}\leq \frac{\sqrt{\ensuremath{\operatorname{vol}} (M_{\phi})}}{(\text{W=width of the fundamental domain})}.$$
Let $u'$ be a harmonic map homotopic to $u$. For each regular $\theta \in
\mathbb{S}^1$, $u^{- 1} (\theta) = S$ and $u^{\prime - 1} (\theta) =
\Sigma_{\theta}$ are homotopic in $M_{\phi}$. Since $M_{\phi}$ is a mapping
torus of $S$, the genus of $\Sigma_{\theta}$ is bigger than the genus of $S$. By Sard's theorem,
$|\chi(\Sigma_{\theta})| \geq |\chi(S)|$ for almost all $\theta$ in $\mathbb{S}^1$. Hence by
equation \eqref{average}, we get
\begin{equation}
  - 4 \pi \chi (S) \leqslant \| \mathrm{d} u' \|_{L^2}  \|- 6\|_{L^2}
  \leqslant 6 \sqrt{\mathrm{vol} (M_{\phi})} \| \mathrm{d} u\|_{L^2} =
  \tfrac{6}{W} \mathrm{vol} (M_{\phi}) . \label{genus}
\end{equation}
Hence we get the genus bound of $S$ 
\[ g \leqslant \frac{3}{4 \pi W} \mathrm{vol} (M_{\phi}) + 1. \]
Once $S$ is fixed and if there is a lower bound for the injectivity radius of $M_\phi$,  the diameter of $S$ is bounded in $M_\phi$. Since $\phi$ identifies $S\times \{0\}$ to $S\times \{1\}$, $||\phi||$ and $W$ are comparable, i.e., there exists $C=C(inj M_\phi, S)$ such that $W\geq C ||\phi||$. Hence the above inequality becomes and prove Theorem \ref{main}
\[ g \leqslant \frac{3}{4 C\pi ||\phi||} \mathrm{vol} (M_{\phi}) + 1. \]

\subsection{Interpretation of $\| \phi \|$ and some applications}

In this section, we interpret $\| \phi \|$ as a quantity comparable to
$\mathrm{ent} (\phi)$ and give an independent proof of Theorem \ref{thm2}.

By Minsky {\cite{minsky-bounded-2001}}, it is known that the infinite cyclic
cover $\tilde{M}_{\phi} = S \times \mathbb{R}$ has a geometric model, the
universal curve over the Teichm{\"u}ller geodesic $\Gamma$ invariant by $\phi$
parametrized by the arc length. More precisely the geometric model is built as
follows. Fix a hyperbolic surface $X_0 = X$ on $\Gamma$. The universal curve
$C_{\Gamma}$ over $\Gamma$ is the collection of $X_t$ where $X_t$ is a
hyperbolic surface at time $t$. The fundamental domain of $\phi$ is the subset
over $[X, \phi (X)]$. The Teichm{\"u}ller distance $d_T (X, \phi (X))$ is
known to be $\mathrm{ent} (\phi)$. Then there exists a biLipschitz map $\Phi :
\tilde{M}_{\phi} \to C_{\Gamma}$ with biLipschitz constant $K$ depending only
on $S$ and the injectivity radius of $M_{\phi}$. Hence the hyperbolic
translation distance $\| \phi \|$ of $\phi$ on $\tilde{M}_{\phi}$ is
comparable to $\mathrm{ent} (\phi) K$, i.e.
\[ K (S, \mathrm{inj} M_{\phi}) \mathrm{ent} (\phi) \leqslant \| \phi \| . \]
By equation \eqref{genus}, we get
\begin{equation}
  \mathrm{ent} (\phi) \leqslant \frac{3}{2 \pi | \chi (S) | K} \mathrm{vol}
  (M_{\phi}) .
\end{equation}
By combining the result of Kojima-McShane {\cite{kojima-normalized-2018}}, we
obtain
\begin{equation}
  \frac{1}{3 \pi | \chi (S) |} \mathrm{vol} (M_{\phi}) \leqslant \mathrm{ent}
  (\phi) \leqslant \frac{3}{2 \pi | \chi (S) | K} \mathrm{vol} (M_{\phi}) .
\end{equation}
One can compare this inequality with the inequality obtained by Brock
{\cite{brock-weil-petersson-2003}} by relating the entropy to Weil-Petersson
translation length. Indeed, the Weil-Petersson metric $g_{WP}$ and the
Teichm{\"u}ller metric $g_T$ satisfy the inequality $g_{WP} \leqslant 2 \pi |
\chi (S) | g_T$ in general, once the injectivity radius of $M_{\phi}$ has a
lower bound, there exists a constant $C$ depending only on the topology of $S$
and the lower bound of the injectivity radius {\cite{kin-entropy-2009}} such
that
\[ C^{- 1} \| \phi \|_{WP} \leqslant \mathrm{ent} (\phi) \leqslant C \| \phi
   \|_{WP} . \]

\appendix\section{Regularity of mixed boundary value problems}

In this appendix, our goal is the existence of mixed boundary value problem
\begin{equation}
  L u = f \text{ in } Q, u = \varphi \text{ on } T \cup B, \tfrac{\partial
  u}{\partial \nu} = \psi \text{ on } F \label{general mvp}
\end{equation}
on a three dimensional cube whose dihedral angles are all $\tfrac{\pi}{2}$.
Here $L$ is defined to be $L u = g^{i j} \partial_i \partial_j u + b^i
\partial_i + c u$ and $g^{i j}$ is the inverse metric of the cube $Q$. To
achieve a solution in $C^{2, \alpha} (\bar{Q})$, we assume that there exists a
function $u' \in C^{2, \alpha} (\bar{Q})$ such that $u' = \varphi$ on $T \cup
B$ and $\tfrac{\partial u'}{\partial \nu} = \psi$ on $F$. This is nothing more
than an easy way of prescribing the compatibility of boundary conditions.

\begin{theorem}
  \label{general existence}There exists a solution $u$ in $C^{2, \alpha} 
  (\bar{Q})$ to \eqref{general mvp}.
\end{theorem}

To show this theorem, the appendix is outline as follows: First, we establish
a maximum principle for mixed boundary value problems. Second, we establish
the $C^{2, \alpha}$ estimates. Using the method of continuity, we solve a
similar mixed boundary value problems on two model domains which are
respectively half ball and a quarter ball. By a classical technique
{\cite{ladyzhenskaya-linear-1968}} we map a neighborhood of a point in a cube
to a neighborhood of the model domains obtaining a $C^{2, \alpha}$ regular
solution for mixed boundary value problem in the cube. The reasons we do not
attempt to solve the problem on a cube directly are: The Green function on a
standard cube is not explicit and is difficult to analyze. The proof basically
follows follows the same lines as classical elliptic theory for Dirichlet
boundary value problems.

First, we recall the existence of a special local coordinate from {\cite[Lemma
2.2]{li-dihedral-2020}}.

\begin{lemma}
  \label{general fermi}If $p \in \partial Q$ belong to $k$ faces $F_1, \ldots,
  F_k$ where $1 \leqslant k \leqslant 3$, then there exists a local coordinate
  system such that $\{z_j \}_{j = 1}^3$ where each face is given by constant
  level set $\{z_j = 0\}$ or $\{z_j = 1\}$ with labeling consistent with the
  coordinate system $\{x_j \}$ given by the diffeomorphism $\Psi$ to the
  standard cube, and on each face $F_i$,
  \begin{equation}
    g (\tfrac{\partial}{\partial z_i}, \tfrac{\partial}{\partial z_j}) = 0, g
    (\tfrac{\partial}{\partial z_i}, \tfrac{\partial}{\partial z_i}) = 1\
    \text{on } F_i \label{crucial boundary condition}
  \end{equation}
  for all $j \neq i$.
\end{lemma}

\subsection{Hopf boundary point lemma}

We have a Hopf boundary point lemma for $u$.

\begin{lemma}
  \label{hopf maximum principle}If $L$ is uniformly elliptic, $c = 0$ and $L u
  \geqslant 0$, let $x_0 \in \partial Q$ be a point at an open edge where
  faces $F_1$ and $F_2$ meet such that $u$ is continuous at $x_0$, $u (x_0) >
  u (x)$ for all $x \in Q$. If $\tfrac{\partial u}{\partial \nu_1} \leqslant
  0$ along $F_1$, then the outer normal derivative of $u$ at $x_0$, if exists
  satisfies the strict inequality
  \begin{equation}
    \tfrac{\partial u}{\partial \nu_2} (x_0) > 0.
  \end{equation}
  If $c \leqslant 0$, the same conclusion holds provided $u (x_0) \geqslant
  0$, and if $u (x_0) = 0$, the same conclusion holds irrespective of the sign
  of $c$.
\end{lemma}

\begin{proof}
  Near $x_0$, there exists a local coordinate $\{z_j \}$ in a neighborhood $U$
  of $x_0$ from Lemma \ref{general fermi}. Without loss of generality, we
  assume that $x_0 = 0$, each $F_i$ is a small piece of $\{z_i = 0\}$ and $U$
  is the intersection of a small ball $B (0, \delta_1)$ with the wedge $W =
  \{z_1 > 0, z_2 > 0\}$. In the local coordinate $\{z_j \}$, we can pick a
  point $y \in F_1 \cap U$ and a ball $B (y, \delta_2)$ such that $x_0 = 0 \in
  \partial B (y, \delta_2)$ and $B (y, \delta_2) \cap W \subset U$.
  
  Setting these things up, we can follow {\cite[Lemma
  3.4]{gilbarg-elliptic-1983}}. For $0 < \rho < \delta_2$, for a constant
  $\alpha > 0$ to be determined, we define the auxiliary function
  \begin{equation}
    v (z) = \mathrm{e}^{- \alpha r^2} - \mathrm{e}^{- \alpha \delta_2^2},
    \label{hopf auxillary function}
  \end{equation}
  where $r$ denotes the $z$-distance between $z$ and $y$. Direct calculation
  shows that
\begin{align}
& L v (z) \\
= & \mathrm{e}^{- \alpha r^2} [4 \alpha^2 g^{i j} (z_i {- y_i} ) (z_j -
y_j) - 2 \alpha (g^{i i} + b^i (z_i - y_i))] + c v \\
\geqslant & \mathrm{e}^{- \alpha r^2} [4 \alpha^2 \lambda (z) r^2 - 2
\alpha (g^{i i} + |b^i |r) + c],
\end{align}
  where $\lambda (z)$ is the smallest eigenvalue of $g^{i j} (z)$. Since the
  domain we are dealing with is compact and $g^{i j}$ is never degenerate,
  hence $\alpha$ could be chosen large enough such that $L v \geqslant 0$ in
  the annular region $A = (B (y, \delta_2)\backslash B (y, \rho)) \cap W$.
  
\begin{figure}[h]
  \begin{center}
    \begin{tikzpicture}[x=1cm,y=1cm]
\draw [thick] (0,0) -- (0,7);
\draw [thick] (0,0) -- (7,0);
\draw [thick, red] (0,0) arc [start angle=270, end angle=450, radius=3];
\draw [thick, blue] (0,1.5) arc [start angle=270, end angle=450, radius=1.5];
\coordinate (y) at (-0.2,3);
\node  at (y)  {$y$};
\node at (0,3) [circle,fill,inner sep=1pt]{};
\node [rotate=40]  at (2,2) {$A=B(y,\delta_2)-B(y,\rho)$};
\node at (6,6) {$W$};
\node at (-0.4,7.3) {$F_1$};
\node at (5.3, -0.4) {$F_2$};
\node  at (-0.1, -0.23) {$x_0$};
\node at (0,0) [circle,fill,inner sep=1pt]{};
\end{tikzpicture}
  \end{center}
  \caption{Illustration of Hopf boundary point lemma}
\end{figure}
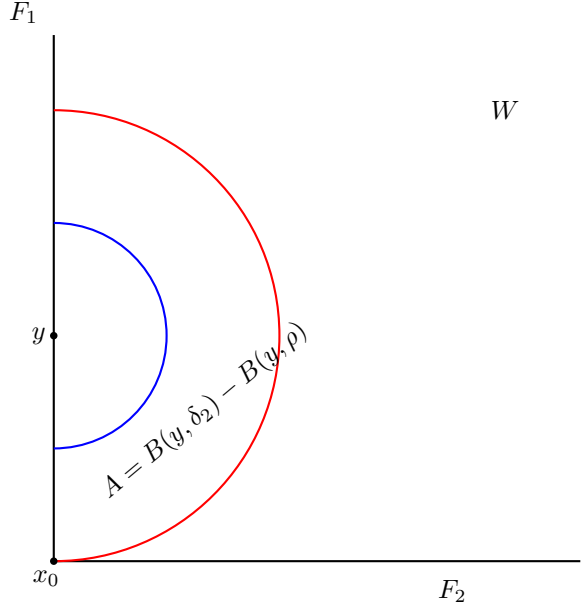
  
  Since $u - u (x_0) < 0$ on $\partial B (y, \rho)$, there is a constant
  $\varepsilon > 0$ for which $u - u (x_0) + \varepsilon v \leqslant 0$ on
  $\partial B (y, \rho)$. The inequality is also satisfied on $\partial B (y,
  \delta_2)$ since $v = 0$ along $\partial B (y, \delta_2)$. So we have that
  \begin{equation}
    L (u - u (x_0) + \varepsilon v) \geqslant - c u (x_0) \geqslant 0 \text{
    in } A.
  \end{equation}
  Note that $\nu_1$ coincides with the Euclidean ($z$ coordinate) normal from
  \eqref{crucial boundary condition}. So we have that $\tfrac{\partial (u - u
  (x_0) + \varepsilon v)}{\partial \nu_1} \leqslant 0$ from $\tfrac{\partial
  u}{\partial \nu_1} \leqslant 0$ along $F_1$ and \eqref{hopf auxillary
  function}. By the classical Hopf boundary point lemma, the maximum of $u - u
  (x_0) + \varepsilon v$ cannot be attained at $F_1 \cap \bar{A}$. So by weak
  maximum principle, $u - u (x_0) + \varepsilon v \leqslant 0$ in $A$, so
  \begin{equation}
    \tfrac{\partial u}{\partial \nu_2} (x_0) \geqslant - \varepsilon
    \tfrac{\partial v}{\partial \nu_2} = - \varepsilon v' (\delta_2) > 0
  \end{equation}
  as claimed.
\end{proof}

It follows directly that the solution $u \in C^2 (\bar{Q})$ to \eqref{mvp}
cannot attain its maximum at open vertical edges. This could be shown via an
alternative argument. If the maximum is attained at $x'$ of an open vertical
edge. Then locally we even reflect $u$ across a neighboring side face using
coordinate system at $x'$ from Lemma \ref{general fermi}. The reflected $u$ is
easily verified to be $C^2$ in the interior. Then the reflected $u$
contradicts the Hopf maximum principle. From the usual Hopf maximum principle,
the maximum cannot be attained in the interior of $Q$ and at open side faces.
To conclude, we have $0 \leqslant u \leqslant 1$ if $u \in C^1 (\bar{Q}) \cap
C^2 (\bar{Q})$.

\subsection{Schauder estimate}

Let $W_k = \{x \in \mathbb{R}^3 : x_i \geqslant 0 \text{for all } 1 \leqslant
i \leqslant k\}$ when $1 \leqslant k \leqslant 3$ and $W_k =\mathbb{R}^3$ when
$k = 0$. Let $\{x_i = 0\}$ be a face, denote by $\partial_{(i)}$ the partial
derivatives on $\{x_i = 0\}$ which is with respect to $x_j$ for $j \neq i$. We
define the seminorms for $u$ along $\{x_i = 0\}$ using $\partial_{(i)}$ i.e.
\begin{equation}
  | \partial_{(i)} \partial_i g|_{\alpha}  \text{and } | \partial_{(i)}
  \partial_{(i)} g|_{\alpha} \label{holder seminorms one face}
\end{equation}
for $g \in C^{2, \alpha} (\{x_i = 0\})$ and $g \in C^{2, \alpha} (\{x_i =
0\})$.

\begin{lemma}
  \label{simon lemma}Any $u \in C^{2, \alpha} (\bar{W}_k)$ satisfies the
  estimate
  \begin{equation}
    | \partial \partial u|_{\alpha} \leqslant C| \Delta u|_{\alpha}
    +\mathcal{E}_k (u) \label{schauder baby}
  \end{equation}
  where $\Delta$ is the standard Laplace operator and $\mathcal{E}_k (u)$
  contains seminorms on faces of $W_k$ in either form of \eqref{holder
  seminorms one face}.
\end{lemma}

\begin{proof}
  This could be proved following the method of {\cite{simon-schauder-1997}}.
  First, one should establish suitable Liouville theorem for harmonic
  functions with homogeneous Dirichlet and Neumann boundary condition along
  faces of $\partial W$. Let $v$ be such a harmonic function with the growth
  rate
  \[ \sup_{B_r \cap \bar{W}_k} |v| \leqslant Cr^{2 + \varepsilon} \]
  for $\varepsilon \in (0, 1)$. One can prove that $v$ is a quadratic
  polynomial. This is done via reflection. By odd (or Schwarz) reflection
  across faces where $v$ vanishes, even reflection across faces where $v$ has
  vanishing normal derivative, one can reduce to the standard Liouville
  theorem. Then one can use the method of scaling to prove \eqref{schauder
  baby}. See {\cite[Theorem 4]{simon-schauder-1997}}.
\end{proof}

We give a local boundary estimate of harmonic functions for some special balls
centered at points of $\partial Q$.

We write the equation $\Delta_g u = 0$ in a local coordinate, we have
\begin{equation}
  \tfrac{1}{\sqrt{g}} \partial_i  (g^{ij}  \sqrt{g} \partial_j u) = g^{ij}
  \partial_i \partial_j u + \tfrac{1}{\sqrt{g}} \partial_i  (g^{ij}  \sqrt{g})
  \partial_j u = : g^{ij} u_{ij} + b_i u_i = 0.
\end{equation}
The coefficient of the second order is given by $g^{ij}$. The local coordinate
system in Lemma \ref{general fermi} then plays the same role as flattening the
boundary in theory of Dirichlet boundary value problems of elliptic equations
if $p$ lies in the open face.

Moreover, it is important in Neumann boundary value problem as well. We
illustrate this by the simplest example $\Delta_g v = f$ in $\mathbb{R}^3_+$
and $\tfrac{\partial v}{\partial \nu} = 0$ along $\partial \mathbb{R}_+^3$.
Here we assume $\partial \mathbb{R}^3_+ = \{x_3 = 0\}$ and temporarily that
$g^{ij}$ are defined on all of $\mathbb{R}^3$. In order to get estimates for
this variable coefficient equation, we have to get estimates for constant
coefficient equation at the boundary $g^{ij} (0) v_{ij} = f'$ and $g^{ij} (0)
v_i \nu_j (0) = 0$. This is the standard procedure of freezing coefficients.
In order to apply the estimates \eqref{schauder baby}, the condition
\eqref{crucial boundary condition} with $i = 3$ is exactly what we need. The
coefficient $g^{ij} (0)$ may not be $\delta^{ij}$ for $1 \leqslant i, j
\leqslant 2$, but one can make a linear change of variables.

Let $z_0$ be a point at the boundary $\partial Q$. In a neighborhood of $z_0$,
there exists a local coordinate system satisfying properties in Lemma
\ref{general fermi}. We define balls using the $\{z_j \}$ coordinates, that is
\begin{equation}
  B_r (z_0) = \{z : |z - z_0 | \leqslant r\}\  \text{for } r > 0.
\end{equation}
We write short $B_r$ for $B_r (z_0)$.

\begin{lemma}
  \label{c2a estimate}We assume that $\bar{B}_{2 \tau}$ does not intersect top
  face $\bar{T}$; and if $z_0$ is on an open face, then $B_{2 \tau}$ does not
  intersect with any edges; if $z_0$ is on an open edge, then $B_{2 \tau}$
  does not contain any vertices. If $u \in C^{2, \alpha} (\bar{Q})$ solves
  \eqref{general mvp} in $B_{2 \tau}$, then
\begin{align}
& \|u\|_{C^{2, \alpha}  (B_{\tau} \cap Q)} \\
\leqslant & C \|u\|_{C^2  (B_{2 \tau} \cap Q)} + |L u|_{C^{\alpha} (B_{2
\tau} \cap Q)} \\
& + C \| \varphi \|_{C^{2, \alpha} (B_{2 \tau} \cap \partial Q)} + C \|
\psi \|_{C^{1, \alpha} (B_{2 \tau} \cap \partial Q)} \label{controlled by
c2} .
\end{align}
\end{lemma}

\begin{proof}
  Let $v = \xi u$, where $\xi$ vanishes outside the ball $B_{2 \tau}$. So by
  earlier estimate \eqref{schauder baby}, we have that
  \begin{equation}
    | \partial \partial v|_{\alpha} \leqslant C|g^{ij} (z_0) v_{ij} |_{\alpha}
    + C\mathcal{E} (v) .
  \end{equation}
  We analyze the contribution of $\mathcal{E} (v)$ from top or bottom faces
  first. The contribution from the top face for example is
  \begin{equation}
    \mathcal{E} (v) = \| \xi u\|_{C^{2, \alpha} (T \cap B_{2 \tau})} = \| \xi
    \varphi \|_{C^{2, \alpha} (T \cap B_{2 \tau})} \leqslant C \| \varphi
    \|_{C^{2, \alpha} (\ensuremath{\operatorname{spt}} \xi \cap \partial Q)} .
  \end{equation}
  The contribution from the bottom face is of the same form. And along side
  faces for instance $\{z_1 = 0\}$, $\tfrac{\partial u}{\partial z_1} = \psi$
  and
\begin{align}
& | \partial_{(1)} \partial_1 (\xi u) |_{\alpha} \\
= & | \partial_{(1)} (\partial_1 \xi u) + \partial_{(1)} (\xi \psi)
|_{\alpha} \\
\leqslant & C \| \psi \|_{C^{1, \alpha} (\ensuremath{\operatorname{spt}}
\xi \cap \partial Q)},
\end{align}
  by choosing a special $\xi$ such that $\partial_1 \xi = 0$. A standard
  cutoff would suffice for this purpose. So
  \begin{equation}
    | \partial \partial v|_{\alpha} \leqslant C|g^{ij} (z_0) v_{ij} |_{\alpha}
    + C \|u\|_{C^2  (\mathrm{spt} \xi \cap Q)} . \label{schauder baby 1}
  \end{equation}
  Freezing the coefficients, we have that
  \begin{equation}
    g^{ij} (z_0) v_{ij} = Lv - (g^{ij} - g^{ij} (z_0)) v_{ij} - b_i v_i - c v.
    \label{schauder freezing}
  \end{equation}
  Note that the steps above differ very little from standard methods of
  obtaining Schauder estimates except that we have to keep track of the
  special property of the local coordinate so that we can make use of
  \eqref{schauder baby}. See Lemma \ref{general fermi}. Using that
  $|fh|_{\alpha} \leqslant \|f\|_{L^{\infty}} |h|_{\alpha} + |f|_{\alpha}
  \|h\|_{L^{\infty}}$, and that $g_{ij} \in C^{2, \alpha} (\bar{B}_{2 \tau})$,
  we have that
  \begin{equation}
    | (g^{ij} - g^{ij} (z_0)) v_{ij} |_{\alpha} \leqslant C \tau^{\alpha} |
    \partial \partial v|_{\alpha} + C \| \partial \partial v\|_{L^{\infty} 
    (B_{2 \tau} \cap Q)},
  \end{equation}
  and
  \begin{equation}
    |b^i v_i |_{\alpha} \leqslant C \tau^{\alpha}  \|v_i \|_{L^{\infty}} +
    |v_i |_{\alpha} \|b_i \|_{L^{\infty}} \leqslant C\|v\|_{C^{1, \alpha}
    (B_{2 \tau} \cap Q)} \leqslant C\|v\|_{C^2  (B_{2 \tau} \cap Q)},
  \end{equation}
  and
  \begin{equation}
    |c v|_{\alpha} \leqslant C \|v\|_{L^{\infty} (B_{2 \tau} \cap Q)} + C
    |v|_{C^{\alpha} (B_{2 \tau} \cap Q)} .
  \end{equation}
  By choosing $\tau$ small, so that we can absorb $C \tau^{\alpha}  | \partial
  \partial v|_{\alpha}$ and combining with \eqref{schauder baby 1} and
  \eqref{schauder freezing}, we have that
  \begin{equation}
    | \partial \partial v|_{\alpha} \leqslant C|Lv|_{\alpha} + C \|v\|_{C^2 
    (B_{2 \tau} \cap Q)}, \label{schauder with c2}
  \end{equation}
  where $C$ only also depends on $\alpha$ and $\tau$. Then pick $\xi$ to be
  $\xi = \phi (|x - y|)$ with $\phi = 1$ in $[0, \tau]$, $\phi = 0$ in $[2
  \tau, \infty)$, $\tau | \phi' | + \tau^2 | \phi'' | \leqslant C$, then the
  above estimate gives \eqref{controlled by c2} since $Lu = f$ and
  \begin{equation}
    L v = \xi L u + 2 g^{i j} \partial_i \xi \partial_j u + u (g^{i j}
    \partial_i \partial_j \xi + b^i \partial_i \xi) .
  \end{equation}
  
\end{proof}

Observing the proof above does not yet immediately give a real local boundary
estimate, because there is a special requirement in Lemma \ref{c2a estimate}
on the location of the center $x_0$ of the ball. We have maybe missed some
part of the boundary. We have to show that all balls satisfying the
requirement of Lemma \ref{c2a estimate} still covers $\partial Q$, and hence a
neighborhood of $\partial Q$.

We can cover the vertices first, then along each edge, we only have to cover a
\ segment (not including any vertices) shorter than the entire edge with such
balls. The number of balls to cover all edges we used is finite by
compactness, hence there is a lower bound of their radius. Therefore, we only
have to cover a smaller piece on each face. After this process, we have
covered $\partial Q$. We decrease the radius of each ball if needed. This
covering argument also removes dependence of $C$ on $\tau$ in
\eqref{controlled by c2}. We obtain a global Schauder estimate.

\begin{proposition}
  \label{schauder}If $u \in C^{2, \alpha} (\bar{Q})$ solves \eqref{general
  mvp}, then
  \begin{equation}
    \|u\|_{C^{2, \alpha} (\bar{Q})} \leqslant |L u|_{C^{\alpha} (\bar{Q})} + C
    \|u\|_{L^{\infty} (\bar{Q})} + C \| \varphi \|_{C^{2, \alpha} (\partial
    Q)} + C \| \psi \|_{C^{1, \alpha} (\partial Q)} \label{schauder global}
  \end{equation}
\end{proposition}

\begin{proof}
  By patching \eqref{controlled by c2} up with interior Schauder estimates we
  have
  \begin{equation}
    \|u\|_{C^{2, \alpha} (\bar{Q})} \leqslant C \|u\|_{L^{\infty} (\bar{Q})} +
    C\|u\|_{C^2 (\bar{Q})} + C \| \varphi \|_{C^{2, \alpha} (\partial Q)} + C
    \| \psi \|_{C^{1, \alpha} (\partial Q)}
  \end{equation}
  Applying the global interpolation inequality
  \begin{equation}
    \|u\|_{C^2 (\bar{Q})} \leqslant \varepsilon \|u\|_{C^{2, \alpha}
    (\bar{Q})} + C_{\varepsilon} \|u\|_{L^{\infty} (\bar{Q})}
  \end{equation}
  and choosing $\varepsilon > 0$ sufficiently small, we can eliminate
  $\|u\|_{C^2 (\bar{Q})}$ and get \eqref{schauder global}.
\end{proof}

\begin{remark}
  It is possible to bound $\|u\|_{C^{2, \alpha} (B_{2 \tau} \cap Q)}$ by
  $\sup_{B_{2 \tau} \cap Q} |u|$, $|L u|_{C^{\alpha} (B_{2 \tau} \cap
  \bar{Q})}$, $\| \varphi \|_{C^{2, \alpha} (B_{2 \tau} \cap \partial Q)}$ and
  $C \| \psi \|_{C^{1, \alpha} (B_{2 \tau} \cap \partial Q)}$.
\end{remark}

\subsection{Solvability on model domains}

Let the (open) half ball be $B_+ = \{x : |x| < 1, x_1 > 0\}$. We solve
\begin{equation}
  \Delta u = f \text{ in } B_+, \tfrac{\partial u}{\partial x_1} = f_1 \text{
  on } \{|x| \leqslant 1, x_1 = 0\}, u = f_2 \text{ on } \{|x| = 1, x_1
  \geqslant 1\} . \label{mvp half ball}
\end{equation}
Without loss of generality, we assume that $f_2 = 0$. Let $G (x, y)$ be the
Green function of the three dimensional unit ball, then by reflection method,
we can write down explicitly that
\begin{equation}
  G (x, y) = - \frac{1}{4 \pi |x - y|} + \frac{1}{4 \pi |y|  |x - \bar{y} |}
\end{equation}
where $\bar{y} = \tfrac{y}{|y|^2}$ is the spherical reflection with respect to
the unit sphere and denote
\begin{equation}
  \tilde{y} = (- y_1, y_2, y_3), \hat{y} = (y_1, - y_2, y_3) .
\end{equation}
Let $\tilde{G} (x, y) = G (x, y) + G (x, \tilde{y})$, so $\Delta_x \tilde{G}
(x, y) = \delta (x - y)$, and $\tfrac{\partial \tilde{G}}{\partial x_1} = 0$
along $B \cap \{x_1 = 0\}$.

\begin{lemma}
  \label{half ball mvp solvability}We have that the solution $u$ to \eqref{mvp
  half ball} on half ball can be represented as
  \begin{equation}
    u (x) = \int_{B +} \tilde{G} (x, y) f (y) \mathrm{d} y + \int_{B \cap
    \{x_1 = 0\}} \tilde{G} (x, y) f_1 (y) \mathrm{d} y, \label{half ball
    explicit solution}
  \end{equation}
  and $u \in C^{2, \alpha} (\bar{B}_+)$.
\end{lemma}

\begin{proof}
  First, the solution \eqref{half ball explicit solution} follows easily from
  Green's formula. We only have to show the $C^{2, \alpha}$ regularity of $u$.
  Performing an even reflection of $f$ across the plane $\{x_1 = 0\}$, we get
  a reflected $\tilde{f} \in C^{\alpha} (\bar{B})$. By a change of variables,
\begin{align}
& \int_{B +} \tilde{G} (x, y) f (y) \mathrm{d} y \\
= & \int_{B_+} G (x, y) f (y) \mathrm{d} y + \int_{B_-} G (x, y) f
(\tilde{y}) \mathrm{d} y \\
= & \int_B G (x, y) \tilde{f} (y) \mathrm{d} y \in C^{2, \alpha}
(\bar{B}_{+, +}) .
\end{align}
  Since $g$ vanishes on $\{|x| = 1, x_1 = 0\}$, extend $f_1$ to all of $\{x_1
  = 0\}$ by setting $\bar{f}_1 (x)$ to be $- |x|^{- 1} f_1 (\tfrac{x}{|x|^2})$
  if $|x| \geqslant 1$ (This is the Kelvin transform in dimension three). We
  denote the extended version by $\bar{f}_1$. It is not difficult to verify
  that $\bar{f}_1 \in C^{1, \alpha} (\{x_1 = 0\})$. Indeed, $\bar{f}_1$
  vanishes along $\{|x| = 1, x_1 = 0\}$ since $f_1$ has to satisfy the
  compatibility condition, so we have $\bar{f}_1$ is continuous. By direct
  calculation if $|x| \geqslant 1$,
  \begin{equation}
    x^i \partial_i \bar{f}_1 = |x|^{- 1} f_1 (\tfrac{x}{|x|^2}) + |x|^{- 3}
    x^i \partial_j f_1 (\tfrac{x}{|x|^2}) \label{radial derivative kelvin
    transform}
  \end{equation}
  We see that $x^i \partial_i \bar{f}_1 = x^i \partial_i f_1$ on $\{|x| = 1,
  x_1 = 0\}$. And derivatives of $\bar{f}_1$ along tangential direction of
  $\{|x| = 1, x_1 = 0\}$ vanishes, this implies that $\bar{f}_1 \in C^1 (\{x_1
  = 0\})$. The H{\"o}lder continuity of $x^i \partial_i \bar{f}_1$ 
  readily follows from
  \begin{equation}
    |x^i \partial_i f_1 (x) - z^i \partial_i f_1 (z) | \leqslant C |x -
    z|^{\alpha}
  \end{equation}
  for each $z \in \{|x| = 1, x_1 = 0\}$, $x \in \{|x| \leqslant 1, x_1 = 0\}$
  and \eqref{radial derivative kelvin transform}. Other derivatives are
  similar proved. We use again a change of variables, and
\begin{align}
& \int_{\{|y| \leqslant 1, y_1 = 0\}} \tilde{G} (x, y) f_1 (y) \mathrm{d}
y \\
= & \int_{\{|y| \leqslant 1, y_1 = 0\}} \psi (x, y) f_1 (y) \mathrm{d} y
\\
& + \int_{\{|y| \leqslant 1, y_1 = 0\}} (\dfrac{1}{4 \pi |y| |x - \bar{y}
|} + \dfrac{1}{4 \pi | \tilde{y} | |x - \overline{\tilde{y}} |}) f_1 (y)
\mathrm{d} y \\
= & \int_{\{|y| \leqslant 1, y_1 = 0\}} \psi (x, y) f_1 (y) \mathrm{d} y
\\
& + \int_{\{|y| \geqslant 1, y_1 = 0\}} (\dfrac{1}{4 \pi | \bar{y} | |x -
y|} + \dfrac{1}{4 \pi | \bar{y} | |x - \tilde{y} |}) f_1 (\bar{y})
(\tfrac{1}{|y|^2} \mathrm{d} y) \\
= & \int_{\{y_1 = 0\}} \psi (x, y) \bar{f}_1 (y) \mathrm{d} y.
\end{align}
  We conclude from Schauder regularity of Neumann problem for half space that
  this contribution of $u$ is also in $C^{2, \alpha} (\bar{B}_+)$. Here
  \[ \psi (x, y) = - \tfrac{1}{4 \pi |x - y|} - \tfrac{1}{4 \pi |x - \tilde{y}
     |} \]
  is the Neumann Green function for the half space, and we have used that
  Kelvin transform and reflection across $\{y_1 = 0\}$ commutes, and that $|
  \tilde{y} | = |y|$. 
\end{proof}

The problem \eqref{mvp half ball} illustrates the reflection principle of how
to handle the mixed boundary value problems. Now we turn to a slightly more
complicated model. Let $B_{+, +}$ be the (open) quarter ball
\[ \{x : |x| < 1, x_1 > 0, x_2 > 0\}, \]
and $\Sigma_1 = \{|x| \leqslant 1, x_1 = 0, x_2 \geqslant 0\}$, $\Sigma_2 =
\{|x| \leqslant 1, x_1 \geqslant 0, x_2 = 0\}$. Again, using reflection, we
have,

\begin{lemma}
  \label{solvability quarter ball euclidean}Assume that there exists a
  function $v \in C^{2, \alpha} (\bar{B}_{+, +})$ with $\tfrac{\partial
  v}{\partial x_i} = f_i$ on $\Sigma_i$ and $v = f_3$ on $\{|x| = 1, x_1
  \geqslant 0, x_2 \geqslant 0\}$. There exists a unique solution $u$ in
  $C^{2, \alpha} (\bar{B}_{+, +})$.
  \begin{equation}
    \Delta u = f \text{ in } B_{+, +}, \tfrac{\partial u}{\partial x_i} = f_i
    \text{ on } \Sigma_i, u = f_3 \text{ on } \{|x| = 1, x_1 \geqslant 0, x_2
    \geqslant 0\} . \label{mvp parametrized by t under delta}
  \end{equation}
\end{lemma}

\begin{proof}
  The boundary of the quarter ball has two point corners $\mathbf{v}_{\pm} =
  \{x_3 = \pm 1, x_1 = x_2 = 0\}$ and three edge corners which are
  $\mathbf{E}_1 = \{x_2 = 0, x_1 \geqslant 0, |x| = 1\}$, $\mathbf{E}_1 =
  \{x_1 = 0, x_2 \geqslant 0, |x| = 1\}$ and $\mathbf{E}_3 = \{- 1 \leqslant
  x_3 \leqslant 1, x_1 = 0, x_2 = 0\}$.
  
  By subtracting from $u$ a $C^{2, \alpha} (\bar{B}_{+, +})$ function, we can
  assume that $f_3 = 0$. The Green function to the problem is
  \begin{equation}
    \hat{G} (x, y) = \tilde{G} (x, y) + \tilde{G} (x, \hat{y}) .
  \end{equation}
  We can perform even reflection on $f$ two times, we can subtract $\int_B G
  (x, y) f (y) \mathrm{d} y$ from $u$, so we are reduce to the case that $f =
  0$.
  
  We extend $f_1$ to $\{x_1 = 0, x_2 \geqslant 0\}$ and $f_2$ to $\{x_1
  \geqslant 0, x_2 = 0\}$ similarly as Lemma \ref{half ball mvp solvability} \
  using the Kelvin transform. It is easily checked that $\partial_2 f_1 =
  \partial_1 f_2$ along $\{x_1 = x_2 = 0\}$. Let $v$ be a function such that
  $\tfrac{\partial w}{\partial x_1} = f_1$, then we see that $u - v$ satisfies
  the equation
\begin{align}
\Delta (u - w) & = - \Delta w \text{ in } \{x_1 \geqslant 0, x_2 \geqslant
0\}, \\
\tfrac{\partial (u - w)}{\partial x_1} & = 0 \text{ on } \{x_1 = 0, x_2
\geqslant 0\}, \\
\tfrac{\partial (u - w)}{\partial x_2} & = f_2 - \tfrac{\partial
w}{\partial x_2} \text{ on } \{x_1 \geqslant 0, x_2 = 0\} .
\end{align}
  We solve separately by reflection technique
\begin{align}
\Delta w_1 & = - \Delta w \text{ in } \{x_1 \geqslant 0, x_2 \geqslant
0\}, \\
\tfrac{\partial w_1}{\partial x_1} & = 0 \text{ on } \{x_1 = 0, x_2
\geqslant 0\}, \\
\tfrac{\partial w_1}{\partial x_2} & = 0 \text{ on } \{x_1 \geqslant 0,
x_2 = 0\},
\end{align}
  and
\begin{align}
\Delta w_2 & = 0 \text{ in } \{x_1 \geqslant 0, x_2 \geqslant 0\}
\\
\tfrac{\partial w_2}{\partial x_1} & = 0 \text{ on } \{x_1 = 0, x_2
\geqslant 0\} \\
\tfrac{\partial w_2}{\partial x_2} & = f_2 - \tfrac{\partial v}{\partial
x_2} \text{ on } \{x_1 \geqslant 0, x_2 = 0\} .
\end{align}
  We see then $u = v + w_1 + w_2$ by linearity and it is easily verified to be
  in $C^{2, \alpha} (\bar{B}_{+, +})$.
\end{proof}

Now we are ready to establish the following more general existence theorem.

We assume that $g$ is a metric on the quarter ball $B_{+, +}$ satisfying that
$g^{11} = 1$, $g^{12} = g^{13} = 0$ on $\Sigma_1$ and $g^{22} = 1$, $g^{21} =
g^{23} = 0$ on $\Sigma_2$, and the spherical piece of $\partial B_{+, +}$ meet
other pieces $\Sigma_i$ with a constant contact angle $\tfrac{\pi}{2}$. This
condition is for the sake of compatibility of two Neumann boundary conditions
along $\{|x| \leqslant 1, x_1 = 0, x_2 = 0\}$ as we shall see later. This is
sufficient for our use, see {\cite[Lemma 2.2]{li-dihedral-2020}}.

\begin{lemma}
  \label{mvp L}There exists a unique solution $u \in C^{2, \alpha} (B_{+, +})$
  if $c \leqslant 0$ to the mixed boundary value problem
  \begin{equation}
    L u = f \text{ in } B_{+, +}, \tfrac{\partial u}{\partial x_i} = f_i
    \text{ on } \Sigma_i, u = f_3 \text{ on } \{|x| = 1, x_1 \geqslant 1\} .
    \label{mvp problem quarter ball L}
  \end{equation}
\end{lemma}

\begin{proof}
  As before, we can assume that $f_3 = 0$. We use the method of continuity.
  Denote $g_0$ be the metric $g$ and let $S$ be the set of all $t \in [0, 1]$
  such that the problem parametrized by $t$
  \begin{equation}
    L_t u = f \text{ in } B_{+, +}, \tfrac{\partial u}{\partial x_i} = f_i
    \text{ on } \Sigma_i, u = 0 \text{ on } \{|x| = 1, x_1 \geqslant 1\}
    \label{mvp parametrized by t}
  \end{equation}
  is solvable where the operators $L_t = (1 - t) \Delta + t L$. We define the
  metric $g_t$ to be the inverse of $t g^{i j} + (1 - t) \delta^{i j}$ and we
  see that the assumptions on $g$ are preserved by such a path $g_t$.
  
  We see first that $0 \in S$ by Lemma \ref{solvability quarter ball
  euclidean}. From Schauder estimate \eqref{schauder global}, the set $S$ is
  closed. It remains to show that $S$ is open.
  
  We fix $t_0 \in S$. Then at $t$, for every $z \in C^{2, \alpha} (\bar{B}_{+,
  +})$ with $z = 0$ on $\{|x| = 1, x_1 \geqslant 1\}$, we define the following
  problem
  \begin{equation}
    L_{t_0} v = f + L_{t_0} z - L_t z \text{ in } B_{+, +}, \tfrac{\partial
    v}{\partial x_i} = f_i \text{ on } \Sigma_i, v = 0 \text{ on } \{|x| = 1,
    x_1 \geqslant 1\} .
  \end{equation}
  By assumption that $t_0 \in S$, there exists a solution in $C^{2, \alpha}
  (\bar{B}_{+, +})$ which denote by $v = A z$. We wish to find a fixed point
  $v = A v$, which gives a solution of \eqref{mvp parametrized by t} in the
  metric $g_t$.
  
  We show that $A$ is a contraction operator for all $t$ with $|t - t_0 | <
  \varepsilon$ where $\varepsilon$ will be determined later. Let $v_i = A z_i$
  where $z_i \in C^{2, \alpha} (\bar{B}_{+, +})$. The difference $v_1 - v_2$
  is a solution to the problem
\begin{align}
L_{t_0} (v_1 - v_2) & = (L_t - L_{t_0}) (z_1 - z_2) \text{ in } B_{+, +},
\\
\tfrac{\partial v}{\partial x_i} & = 0 \text{ on } \Sigma_i, \\
v & = 0 \text{ on } \{|x| = 1, x_1 \geqslant 1\} .
\end{align}
  By Schauder estimates \eqref{schauder global}, we have that
  \begin{equation}
    \|v_1 - v_2 \|_{C^{2, \alpha} (\bar{B}_{+, +})} \leqslant C \|[L_t -
    L_{t_0}] (z_1 - z_2)\|_{C^{\alpha} (\bar{B}_{+, +})} .
  \end{equation}
  Note that
  \begin{equation}
    \|[L_t - L_{t_0}] (z_1 - z_2)\|_{C^{\alpha} (\bar{B}_{+, +})} \leqslant C
    |t - t_0 |  \|z\|_{C^{2, \alpha} (\bar{B}_{+, +})} .
  \end{equation}
  We choose $\varepsilon > 0$ so small that we have
  \begin{equation}
    \|v_1 - v_2 \|_{C^{2, \alpha} (\bar{B}_{+, +})} \leqslant \tfrac{1}{2}
    \|z_1 - z_2 \|_{C^{2, \alpha} (\bar{B}_{+, +})} .
  \end{equation}
  Applying the contraction mapping principle, there exists a unique $u \in
  C^{2, \alpha} (\bar{B}_{+, +})$ such that
  \begin{equation}
    u = A u \in C^{2, \alpha} (\bar{B}_{+, +})
  \end{equation}
  obtaining the openness of $S$.
\end{proof}

\subsection{Existence of $C^{2, \alpha}$ solutions on a cube}

We are now ready to show that the solution on a cube is in $C^{2, \alpha}
(\bar{Q})$. We use a technique from Chapter 3, Section 1 of
{\cite{ladyzhenskaya-linear-1968}} (see also {\cite[Lemma
6.10]{gilbarg-elliptic-1983}}).

Consider a point $x_0 \in \partial Q$, $x_0$ lies at an open edge or at the
vertex. Then near $x_0$, the problem is either mixed Neumann-Neumann, mixed
Dirichlet-Neumann or mixed Dirichlet-Neumann-Neumann type. We assume that
$x_0$ is a vertex at the top face $T$ and will show the following theorem.
Other cases is similar.

\begin{theorem}
  Near $x_0$ the solution $u$ to \eqref{general mvp} is $C^{2, \alpha}$ up to
  the vertices.
\end{theorem}

\begin{proof}
  Note that $u' - u$ satisfies the homogeneous boundary condition. We can
  perform even reflection on $u' - u$ accross side faces two times (see
  {\cite{li-dihedral-2020}}), then apply the weak theory of Dirichlet boundary
  value problems. We see that $u' - u \in C^{\mu} (\bar{Q}) \cap C^{2, \alpha}
  (Q)$ for some $\mu \in (0, 1)$ using weak theory. So ${u \in C^{\mu}} 
  (\bar{Q}) \cap C^{2, \alpha} (Q)$.
  
  Similar to Lemma \ref{general fermi}, there exists a coordinate system in a
  neighborhood $\Omega$ of $x_0$ such that the a piece of the top face near
  $x_0$ is mapped into a piece of spherical boundary $\{|x| = 1, x_1 \geqslant
  0, x_2 \geqslant 0\}$ and side faces are mapped into pieces of planar
  boundaries $\{x_i = 0, |x| \leqslant 1\}$ ($i$ is 1 or 2). And the metric
  satisfies that $g^{11} = 1$, $g^{12} = g^{13} = 0$ on $\Sigma_1$ and $g^{22}
  = 1$, $g^{21} = g^{23} = 0$ on $\Sigma_2$, and the spherical piece of
  $\partial B_{+, +}$ meet other pieces $\Sigma_i$ with a constant contact
  angle $\tfrac{\pi}{2}$. This is the same condition as in Lemma \ref{mvp L}.
  
  We show that $u \in C^{2, \alpha} (\bar{Q})$ by approximation. We may now
  regard $u$ now as a solution to a problem of the form \eqref{mvp problem
  quarter ball L}. We can take a sequence of (compatible) boundary data
  $f_i^{(k)}$ such that they converge uniformly to $f_i$. Let $u^{(k)}$ be the
  solution to
  \begin{equation}
    L u_k = f \text{ in } B_{+, +}, \tfrac{\partial u}{\partial x_i} =
    f_i^{(k)} \text{ on } \Sigma_i, u = f_3^{(k)} \text{ on } \{|x| = 1, x_1
    \geqslant 1\} .
  \end{equation}
  Using the maximum principle Lemma \ref{hopf maximum principle}, we can show
  that $u_k$ converges uniformly to $u$, and by Schauder estimates
  \eqref{schauder global} $u^{(k)} \in C^{2, \alpha} (\bar{B}_{+, +})$,
  $\|u^{(k)} \|_{C^{2, \alpha} (\bar{B}_{+, +})} \leqslant C$ for some
  constant $C > 0$. Combining with Arzela theorem, we have that $u \in C^{2,
  \alpha} (\bar{B}_{+, +})$. Returning to the cube $Q$, hence we have shown
  that near $x_0$ the solution $u$ is $C^{2, \alpha}$ up to the vertices.
\end{proof}

\begin{remark}
  It might be possible that the assertion $u \in C^{\mu} (\bar{Q})$ follows
  directly from a weak theory for mixed boundary value problems. Then the
  $C^{2, \alpha} (\bar{Q})$ regularity follows directly from uniqueness.
\end{remark}

\begin{remark}
  Alternatively, we can further divide the quarter ball by half and prescribe
  a Dirichlet boundary condition on the extra neighboring face of the origin.
  Then near the origin, the boundary condition is mixed
  Dirichlet-Neumann-Neumann type. We can study the problem in a similar
  fashion as in Lemma \ref{mvp L}. The Green function of the model problem is
  obtained via one more reflection. This approach is simply longer in
  presentation and the idea is the same.
\end{remark}

\

\end{document}